\documentclass[11pt]{amsart}
\usepackage[margin=2.5cm]{geometry}
\usepackage{amssymb,amsmath,booktabs,enumitem}
\usepackage[T1]{fontenc}
\usepackage{lmodern,needspace}
\usepackage[dvipsnames]{xcolor}
\usepackage[colorlinks=true,linkcolor=MidnightBlue,citecolor=MidnightBlue,urlcolor=MidnightBlue]{hyperref}
\usepackage{array}

\newtheorem{theorem}{Theorem}[section]
\newtheorem{proposition}[theorem]{Proposition}
\newtheorem{lemma}[theorem]{Lemma}
\newtheorem{corollary}[theorem]{Corollary}

\newtheorem*{mainthm}{Theorem A}
\newtheorem*{thmB}{Theorem B}
\theoremstyle{remark}
\newtheorem{remark}[theorem]{Remark}

\newcommand{\PP}{\mathbb{P}}\newcommand{\ZZ}{\mathbb{Z}}\newcommand{\QQ}{\mathbb{Q}}
\newcommand{\CC}{\mathbb{C}}\newcommand{\FF}{\mathbb{F}}\newcommand{\OO}{\mathcal{O}}
\newcommand{\cC}{\mathcal{C}}\newcommand{\cX}{\mathcal{X}}\newcommand{\fa}{\mathfrak{a}}\newcommand{\fm}{\mathfrak{m}}
\newcommand{\ls}[1]{\lvert #1\rvert}\newcommand{\ob}{\operatorname{ob}}
\newcommand{\Conf}{\operatorname{Conf}}\newcommand{\coker}{\operatorname{coker}}\newcommand{\NF}{\operatorname{NF}}
\newcommand{\mult}{\operatorname{mult}}

\title[Lifting negative curves and Seshadri constants]{Lifting negative curves from finite fields and Seshadri constants at ten points}
\author{Piotr Pokora and Tomasz Szemberg}
\address{Piotr Pokora, Department of Mathematics, University of the National Education Commission Krakow, Podchor\k{a}\.zych 2, PL-30-084 Krak\'ow, Poland.}
\email{piotrpkr@gmail.com}
\address{Tomasz Szemberg, Department of Mathematics, University of the National Education Commission Krakow, Podchor\k{a}\.zych 2, PL-30-084 Krak\'ow, Poland.}
\email{tomasz.szemberg@gmail.com}
\date{September 29, 2026}
\hypersetup{pdftitle={Lifting negative curves from finite fields and Seshadri constants at ten points},pdfauthor={Piotr Pokora and Tomasz Szemberg}}

\begin{document}

\begin{abstract}
We use finite-field calculations and mixed-characteristic deformation to construct nef divisors on very general complex blow-ups of the plane.
A lifting and smoothing criterion converts a reducible nodal divisor over $\FF_{127}$ into a smooth genus-$10$ curve of class $117H-37(E_1+\cdots+E_{10})$, retaining control of its first-order obstruction map.
Its normal bundle in a suitable deformation is semistable, so Petrakiev's multiplicity estimate yields the nef class $27379H-8658(E_1+\cdots+E_{10})$ at ten very general points.
Their multipoint Seshadri constant is therefore at least $8658/27379$, improving Petrakiev's bound $228/721$.
The geometric input is an existence criterion in characteristic zero verified by computations in positive characteristic.
\end{abstract}

\maketitle

\section{Introduction}

The geometry of plane curves with prescribed multiplicities is governed by the interplay between special configurations, where additional curves appear, and very general configurations, where these curves disappear.
This paper uses arithmetic to construct such a special configuration: we find a reducible negative divisor over a finite field, prove that it lifts and smooths in characteristic zero, and use its normal bundle to obtain a nef divisor on a very general complex blow-up.
The resulting improvement in the ten-point Seshadri constant is an application of this construction.
Its essential ingredient is a verifiable criterion for passing from a nodal finite-field example to a smooth complex curve with a prescribed nonzero obstruction to deformation.

\subsection{The Seshadri bound}
Let $p_1,\dots,p_{10}\in\PP^2_\CC$ be very general points, and let $X$ be their blow-up.
Write $H$ for the pull-back of a line, $E_i$ for the exceptional curves, and $E=\sum_{i=1}^{10}E_i$.
We use the normalization
$$
\varepsilon_{10}=\varepsilon(\PP^2,\OO(1);p_1,\dots,p_{10})
=\inf_C\frac{\deg C}{\sum_{i=1}^{10}\mult_{p_i}C}
=\sup\{t\ge0:H-tE\text{ is nef}\},
$$
where the infimum is over integral curves meeting at least one of the points.
The volume bound gives $\varepsilon_{10}\le1/\sqrt{10}$.
Nagata's conjecture, originating in his work on Hilbert's fourteenth problem \cite{Nag}, predicts equality.
More generally, it predicts $d>\sum_i m_i/\sqrt{s}$ for every nonzero plane curve of degree $d$ with multiplicities $m_i$ at $s\ge10$ very general points.
Nagata established the square cases $s\ge16$; the ten-point case is the first nonsquare case.
We refer to \cite{Primer} for background.

\begin{mainthm}
The divisor $27379H-8658E$ is nef on the blow-up of $\PP^2_\CC$ at ten very general points.
Equivalently, every plane curve of degree $d$ with multiplicities $m_i$ at these points satisfies
$$
27379\,d\ge8658\sum_{i=1}^{10}m_i.
$$
In particular,
$$
\varepsilon_{10}\ge\frac{8658}{27379}=0.3162277658059\ldots.
$$
\end{mainthm}

The numerical relation with the preceding bound is particularly simple:
$$
(19+6\sqrt{10})^2=721+228\sqrt{10},\qquad
(19+6\sqrt{10})^3=27379+8658\sqrt{10}.
$$
Thus Theorem~A realizes the third positive solution of $x^2-10y^2=1$ as a nef class, following the class $721H-228E$ obtained by Petrakiev \cite{Pet}.
The gap between the lower bound and $1/\sqrt{10}$ decreases from approximately $3.04\cdot10^{-7}$ to $2.11\cdot10^{-10}$, a factor of approximately $1442$.
The Pell equation identifies the numerical target; the substantive issue is to construct the geometry which makes that target nef.

\subsection{A finite-field construction of the required complex curve}
Set
$$
\Theta=117H-37E.
$$
Then $\Theta^2=-1$, $K_X\cdot\Theta=19$, and $p_a(\Theta)=10$.
The reflection in this class sends $H$ to the class in Theorem~A:
$$
s_\Theta(H)=H+2(H\cdot\Theta)\Theta
=H+234\Theta=27379H-8658E.
$$
Petrakiev's normal-bundle method \cite{Pet}, in its reflection formulation, reduces nefness to an existence problem on a special blow-up.
It is enough to construct a smooth irreducible curve $\Gamma\in\ls\Theta$ and a motion of the ten points along which $\Gamma$ does not lift to first order.
Indeed, a nonsplit extension
$$
0\longrightarrow\OO_\Gamma(\Gamma)
\longrightarrow N_{\Gamma/\cX}
\longrightarrow\OO_\Gamma\longrightarrow0
$$
is automatically semistable when $\Gamma^2=-1$.
We prove this reduction in Section~\ref{sec:refl}.
The reflection viewpoint also appears in the work of Dionne--Roth \cite{DR} and in the recent work of Laface--Ugaglia \cite{LU}.
Here the reflecting curve has genus $10$, and its construction is the main difficulty.

A search over $\FF_{127}$ produces a configuration whose unique divisor in $\ls\Theta$ is reducible.
Its two components have degrees $12$ and $105$, and their strict transforms are smooth curves of genera $0$ and $9$, meeting transversally in two points.
The crucial additional information is infinitesimal: the obstruction map for the total divisor has rank $10$, whereas the combined obstruction map for its two components has rank $12$.
These ranks are computed in the full $20$-dimensional space of ordered configurations, allowing motions which break the symmetry used in the search.

The following statement isolates the passage from this arithmetic example to the required complex curve.
The obstruction map measures the failure of a section of a linear system to lift when its base points move; it is defined explicitly in Proposition~\ref{prop:ss}.

\begin{thmB}[Lifting and smoothing criterion]
Suppose a ten-point blow-up over $\FF_p$ carries a unique divisor $\Gamma_0\in\ls\Theta$, and write $\Gamma_0=\tilde A+\tilde B$.
Assume that $\tilde A$ and $\tilde B$ are smooth geometrically irreducible strict transforms of plane curves, meeting transversally at $k\ge1$ points.
Let their classes be $\alpha$ and $\beta$, and suppose that
$$
h^0(\alpha)=h^0(\beta)=1,\qquad
\operatorname{rank}\ob_\Theta=10,\qquad
\operatorname{rank}(\ob_\alpha,\ob_\beta)=10+k.
$$
Then a complex ten-point blow-up carries a smooth irreducible curve $\Gamma\in\ls\Theta$ with $\operatorname{rank}\ob_\Theta=10$.
Such configurations occur in a family of dimension $10$ in the space of ordered configurations.
\end{thmB}

Theorem~\ref{thm:seed} gives the precise formulation and proof.
The first rank condition makes the effectivity locus of $\Theta$ smooth over $\ZZ_p$, of relative dimension $10$, near the certified configuration.
The second makes the locus on which both components persist smooth of relative dimension $10-k$.
This dimension difference controls the node-smoothing parameters and permits a characteristic-zero deformation which is both smooth and geometrically irreducible.
The obstruction rank persists on that deformation.
Consequently, one finite-field example certifies all the geometry needed for the reflection argument over $\CC$.

The distinction between existence and lifting matters here.
A nonzero section in characteristic $p$ need not survive in characteristic zero, and a reducible divisor need not admit a smoothing in the prescribed linear system.
The obstruction-rank conditions address both issues.
The role of positive characteristic is therefore an arithmetic construction of the required special geometry, followed by a mixed-characteristic deformation argument.
Semistability of symmetric powers, used in the final normal-bundle estimate, is invoked only after passing to characteristic zero.

Our approach was partly inspired by H\`a and Sivakumar's proof of the Chudnovsky--Demailly conjectures \cite{HS}, in which passage to positive characteristic is central to establishing inequalities over a field of characteristic zero.
The connection is methodological: their Frobenius and Hasse-derivative arguments concern symbolic powers of ideals of points, whereas our proof uses a finite-field divisor, determinantal smoothness, and smoothing of nodes.

\subsection{Earlier bounds and geometric context}
Table~\ref{tab:history} places Theorem~A among some of the principal earlier estimates.
All entries use the same normalization of $\varepsilon_{10}$ and are listed in increasing numerical order, rather than as a strict chronology.
In the homogeneous interpolation notation, a bound $\varepsilon_{10}\ge c$ corresponds to the asymptotic inequality $d/m\ge10c$.
Conversely, a uniform homogeneous emptiness bound implies the corresponding inequality for arbitrary multiplicities at very general points by symmetrizing over permutations of the points.
This conversion is important when comparing results stated in different forms.

\begin{table}[htbp]
\centering\small
\begin{tabular}{@{}m{7.0cm}cl@{}}
\toprule
Source & Lower bound for $\varepsilon_{10}$ & Approximate value\\
\midrule
Biran \cite{Bir} & $6/19$ & $0.3157894737$\\
Harbourne--Ro\'e \cite{HR} & $177/560$ & $0.3160714286$\\
Ciliberto--Miranda \cite{CM} & $55/174$ & $0.3160919540$\\
Dumnicki \cite{Dum}; see \cite[\S3]{Primer} & $313/990$ & $0.3161616162$\\
Ciliberto--Dumitrescu--Miranda--Ro\'e \cite{CDMR};\newline Eckl \cite{Eck} & $117/370$ & $0.3162162162$\\
Petrakiev \cite{Pet} & $228/721$ & $0.3162274619$\\
Theorem~A & $8658/27379$ & $0.3162277658$\\
\midrule
Nagata's predicted value & $1/\sqrt{10}$ & $0.3162277660$\\
\bottomrule
\end{tabular}
\caption{Selected lower bounds for the ten-point Seshadri constant, ordered by strength.}\label{tab:history}
\end{table}

These estimates reflect several different ways of controlling plane interpolation.
Biran's construction combines gluing of ample divisors with Cremona transformations \cite{Bir}.
Harbourne--Ro\'e organize the possible obstructing curves into a computational approach to nefness \cite{HR}.
Dumnicki's splitting methods give effective tests for non-speciality \cite{Dum}; the value $313/990$ is recorded in \cite[\S3]{Primer}.
Ciliberto--Miranda prove non-speciality for $d/m\ge174/55$ by degenerating the blown-up plane and modifying components with negative intersection \cite{CM}.
The subsequent emptiness argument of Ciliberto--Dumitrescu--Miranda--Ro\'e and the asymptotic degeneration argument of Eckl yield $117/370$ \cite{CDMR,Eck}.
Petrakiev obtains $228/721$ using degenerations to curves and semistable vector bundles \cite{Pet}.
Our construction uses his multiplicity estimate after producing the reflecting curve by arithmetic deformation.

The work of de Fernex \cite{dFneg,dFcone} provides a complementary structural perspective through negative curves and the Mori cone of very general plane blow-ups.
This perspective concerns the geometry surrounding Nagata's ray, rather than a further numerical entry in Table~\ref{tab:history}.
For our construction it is important to distinguish the two surfaces involved: the genus-$10$ curve of self-intersection $-1$ exists on a special blow-up, whereas its reflected nef class is obtained on a very general blow-up.
The curve itself is not asserted to persist at very general points.

\subsection{Organization}
Section~\ref{sec:refl} explains how a smooth negative curve with a nonzero obstruction map produces a nef class.
Section~\ref{sec:seed} proves the lifting and smoothing criterion.
Section~\ref{sec:cert} supplies the finite-field configuration and specifies the computations verifying the hypotheses.
Section~\ref{sec:proofA} completes the proof of Theorem~A and records a consequence for possible counterexamples to Nagata's conjecture.
Section~\ref{sec:conclusion} discusses the scope of the construction.
The additional symmetry constraints and computational searches are collected in Appendix~\ref{sec:constraints}.
We work over $\CC$ except in the explicitly indicated arithmetic constructions.

\section{From an obstructed negative curve to a nef divisor}\label{sec:refl}

Let us fix the notation. For a configuration $Z=\{z_1,\dots,z_{10}\}$ of distinct points in $\PP^2$ we denote by $X_Z$ the blow-up of $\PP^2$ along $Z$. We write $H$ for the pull-back of a line, $E_i$ for the exceptional curves, $E=E_1+\dots+E_{10}$, and $K=-3H+E$ for the canonical class.

Let $(T,0)$ be a smooth pointed curve, and let $\cX\to T$ be the blow-up of $\PP^2\times T$ along ten disjoint sections $t\mapsto z_i(t)$. Then $N^1(X_t)=\ZZ H\oplus\bigoplus_i\ZZ E_i$ for all $t$. For a smooth curve $\Gamma\subset X_0$ we have the exact sequence
\begin{equation}\label{eq:N}
0\to N_{\Gamma/X_0}\to N_{\Gamma/\cX}\to\OO_\Gamma\to0 .
\end{equation}
The following result is a semistable version of Petrakiev's Basic Lemma \cite[Lemma~2.1]{Pet}; compare also \cite[Theorem~8.1]{DR} and \cite[Lemma~1.4]{LU}.

\begin{proposition}\label{prop:refl}
Let $\Gamma\subset X_0$ be a smooth connected curve of class $\gamma$ such that the bundle $N_{\Gamma/\cX}$ is semistable of slope $-\kappa<0$, so that $\gamma^2=-2\kappa$. If $D$ is nef on $X_0$, then
$$
s_\gamma(D)=D+\frac{D\cdot\gamma}{\kappa}\,\gamma
$$
is nef on $X_t$ for all but countably many $t\in T$, and hence on the blow-up of $\PP^2$ at ten very general points.
\end{proposition}

\begin{proof}
Let $c$ be the class of an effective divisor on the geometric generic fibre of $\cX\to T$. This divisor is defined over a finite extension of $\CC(T)$. Taking the sum of its Galois conjugates, we obtain a divisor defined over $\CC(T)$ whose closure $\cC\subset\cX$ is flat over $T$ and has generic class $ec$ for some integer $e\ge1$. Let us stress that we do not perform any base change here, since a ramified base change would kill the Kodaira--Spencer class and split \eqref{eq:N}.

Denote by $\mu$ the multiplicity of $\cC$ along $\Gamma$. A local equation of $\cC$ defines a non-zero map $\OO_\Gamma(-\cC)\to\operatorname{Sym}^\mu N^\vee_{\Gamma/\cX}$. Since we are in characteristic zero, the target is semistable of slope $\mu\kappa$, and hence $-e\,c\cdot\gamma\le\mu\kappa$. This shows that
$$
s_\gamma(ec)=\bigl(\cC|_{X_0}-\mu\Gamma\bigr)+\Bigl(\mu+\frac{ec\cdot\gamma}{\kappa}\Bigr)\Gamma
$$
is effective on $X_0$. As $s_\gamma$ is a self-adjoint isometry, we get $e\,s_\gamma(D)\cdot c=D\cdot s_\gamma(ec)\ge0$. So $s_\gamma(D)$ is nef on the geometric generic fibre.

Now let $c$ be a class with $s_\gamma(D)\cdot c<0$. The relative Hilbert scheme of curves of class $c$ is proper over $T$ and does not dominate $T$, so it lies over finitely many points of $T$. Since there are only countably many such classes, $s_\gamma(D)$ is nef on $X_t$ for all but countably many $t$.

Finally, for each such class $c$, the configurations on which $c$ is effective form a closed subset of the configuration space. This subset misses the general point of $T$, so it is proper. Therefore $s_\gamma(D)$ is nef at a very general configuration.
\end{proof}

In order to apply Proposition~\ref{prop:refl} we need a computable criterion for semistability. Let $\Gamma\subset X_0$ be the strict transform of a plane curve $f=0$ of degree $d$ with multiplicities $m_i$ at the points $z_i$. We denote by $M(Z)$ the interpolation matrix, that is, the matrix of the jet map from the space of polynomials of degree at most $d$ to $\bigoplus_i\OO/\fm_{z_i}^{m_i}$.

\Needspace{10\baselineskip}
\begin{proposition}\label{prop:ss}
Assume that $\Gamma$ is smooth and $\Gamma^2<0$.
\begin{enumerate}[label=\textup{(\alph*)}]
\item We have $H^1(N_{\Gamma/X_0})\cong H^1(\OO_{X_0}(\Gamma))=\coker M(Z_0)$. If the points move with velocity $v$, then the class of \eqref{eq:N} is $e(v)=-[(\partial_vM)f]$, and $\Gamma$ lifts to first order along $v$ if and only if $e(v)=0$. We write $\ob(v)=[(\partial_vM)f]$ and call $\ob$ the obstruction map.
\item If $\Gamma^2=-1$, then $N_{\Gamma/\cX}$ is semistable if and only if $e(v)\ne0$.
\end{enumerate}
\end{proposition}

\begin{proof}
For (a), since $H^1(\OO_{X_0})=H^2(\OO_{X_0})=0$, the sequence $0\to\OO_{X_0}\to\OO_{X_0}(\Gamma)\to N_{\Gamma/X_0}\to0$ gives $H^1(N_{\Gamma/X_0})\cong H^1(\OO_{X_0}(\Gamma))$. The vanishing of $H^1(\PP^2,\OO(d))$ identifies this group with $\coker M$. A first-order lift is a polynomial $f_1$ with $M(Z_0)f_1=-(\partial_vM)f$, see \cite[\S3.2]{Ser}.

For (b), observe that a saturated line subbundle of $N_{\Gamma/\cX}$ other than $N_{\Gamma/X_0}$ maps non-trivially to $\OO_\Gamma$, so it has degree at most $0$. Since $\deg N_{\Gamma/\cX}=-1$, such a subbundle destabilizes if and only if it is $\OO_\Gamma$, that is, if and only if \eqref{eq:N} splits.
\end{proof}

\begin{corollary}\label{cor:red}
Suppose that some complex ten-point blow-up $X_Z$ carries a smooth irreducible curve $\Gamma\in\ls\Theta$ with $\ob\ne0$. Then $27379H-8658E$ is nef on the blow-up of $\PP^2$ at ten very general points.
\end{corollary}

\begin{proof}
Let us move the points along a line in a direction $v$ with $\ob(v)\ne0$. By Proposition~\ref{prop:ss}(b), the bundle $N_{\Gamma/\cX}$ is semistable of slope $-\tfrac12$. Now Proposition~\ref{prop:refl} applied with $D=H$ and $\kappa=\tfrac12$ gives the nefness of $s_\Theta(H)=27379H-8658E$.
\end{proof}

\section{Lifting and smoothing a finite-field divisor}\label{sec:seed}

This section supplies the arithmetic existence argument underlying Theorem~A.
Theorem~B follows from the criterion below, which separates three requirements: lifting the effectivity locus, smoothing the nodes, and preserving the obstruction rank.
Let $p$ be a prime number and let $\Conf\subset(\mathbb A^2_{\ZZ_p})^{10}$ be the scheme of ten distinct points in the affine chart $z=1$. It is smooth of relative dimension $20$ over $\ZZ_p$.

For a class $c=dH-\sum_im_iE_i$ with $d,m_i\ge0$, the entries of the interpolation matrix $M_c(Z)$ are Hasse derivatives of monomials. Hence they are polynomials with integer coefficients in the coordinates of $Z$. We denote by $L_c\subset\Conf$ the scheme defined by the minors of size equal to the number of columns, using the zero ideal if there are fewer rows; this is the locus where $h^0(c)\ge1$. We write $h^0(c)=\dim\ker M_c$ and $h^1(c)=\dim\coker M_c$. If $h^0(c)=1$, we denote by $\ob_c$ the map from Proposition~\ref{prop:ss}(a), which makes sense over any field.

\begin{theorem}[Arithmetic lifting and smoothing]\label{thm:seed}
Let $Z_0\in\Conf(\FF_p)$ be a configuration, with blow-up $X_0$, satisfying the following conditions.
\begin{enumerate}[label=\textup{(S\arabic*)}]
\item We have $h^0(X_0,\Theta)=1$, and $\ob_\Theta$ has rank $h^1(X_0,\Theta)=10$.
\item The member of $\ls\Theta$ is $\Gamma_0=\tilde A+\tilde B$. Here $A$ and $B$ are absolutely irreducible plane curves whose strict transforms $\tilde A,\tilde B$ are smooth and meet transversally at $k\ge1$ points $n_1,\dots,n_k$. Neither of them contains an exceptional curve.
\item For the classes $\alpha,\beta$ of $\tilde A,\tilde B$ we have $h^0(\alpha)=h^0(\beta)=1$, and the combined map $v\mapsto(\ob_\alpha v,\ob_\beta v)$ has rank $h^1(\alpha)+h^1(\beta)$.
\end{enumerate}
Then there exist ten distinct points $Z\subset\PP^2_\CC$ such that $X_Z$ carries a smooth irreducible curve $\Gamma\in\ls\Theta$ with $\operatorname{rank}\ob_{\Theta}=10$. Moreover, such configurations occur in a family of dimension $10$ in the space of ordered configurations, before quotienting by projective transformations.
\end{theorem}

Let us observe that by the Riemann--Roch theorem
$$
h^1(\alpha)+h^1(\beta)=h^1(\Theta)+\alpha\cdot\beta=10+k ,
$$
since $h^0=1$ and $h^2=0$ for all three classes. Thus condition (S3) requires the maximal possible rank $10+k>10$. The heart of the argument is that the locus of configurations where the two components survive has dimension $10-k$, which is strictly smaller than the dimension $10$ of the locus where $\Theta$ is effective.

\subsection{Smoothness of the effectivity loci}

\begin{lemma}\label{lem:smooth}
Let $c_1,\dots,c_s$ be interpolation classes with $h^0(X_0,c_i)=1$ and $h^1(X_0,c_i)=r_i$, and suppose that the combined obstruction map has rank $r=\sum_ir_i$. Then, near $Z_0$, the scheme $\bigcap_iL_{c_i}$ is smooth over $\ZZ_p$ of relative dimension $20-r$. Moreover, at each of its points we have $h^0(c_i)=1$, and $\ob_{c_i}$ has rank $r_i$.
\end{lemma}

\begin{proof}
Let $\OO$ be the local ring of $\Conf$ at $Z_0$, and write $M=M_{c_i}$ in block form with $M_{11}$ invertible over $\OO$ of size $\operatorname{rank}M(Z_0)$. Row and column operations bring $M$ to the form $\operatorname{diag}(1,S)$, where
$$
S=M_{22}-M_{21}M_{11}^{-1}M_{12}
$$
is the Schur complement, a column of length $r_i$. The maximal minors of $M$ generate the ideal of the entries of $S$. Put $f=(-M_{11}^{-1}M_{12},1)^t$ and $W=(-M_{21}M_{11}^{-1},1)$, so that $S=WMf$. At $Z_0$ we have $Mf=0$ and $WM=0$, hence $d_vS=W(\partial_vM)f=\ob_{c_i}(v)$ in the basis $W$ of the left kernel. Therefore the $r$ Schur entries have linearly independent relative differentials modulo $p$. Together with $p$ they form a part of a regular system of parameters, and the claim follows from the Jacobian criterion.
\end{proof}

By (S1), (S3) and Lemma~\ref{lem:smooth}, near $Z_0$ the scheme $L_\Theta$ is smooth of relative dimension $10$ and $\ob_\Theta$ has rank $10$ on it, while $L_\alpha\cap L_\beta\subset L_\Theta$ is smooth of relative dimension $10-k$. Let $R=\widehat\OO_{L_\Theta,Z_0}\cong\ZZ_p[[u_1,\dots,u_{10}]]$, and let $\fa\subset R$ be the ideal of $L_\alpha\cap L_\beta$; in particular $\fa\ne0$.

Over $\operatorname{Spec}R$, the kernel line of $M_\Theta$ gives a section of $\OO(\Theta)$ on the blown-up family, which is non-zero on every fibre. Its zero scheme $\cC$ is flat and proper over $R$. Its special fibre is $\Gamma_0$: no exceptional curve splits off, since $\mult_zA+\mult_zB=37$ at every point $z\in Z_0$. By (S2), $\Gamma_0$ is a nodal curve with nodes $n_1,\dots,n_k$ and connected dual graph.

\subsection{The reducible locus and the node-smoothing parameters}

\begin{lemma}\label{lem:spec}
Let $D$ be a complete discrete valuation ring with residue field finite over $\FF_p$, and let $Z\colon\operatorname{Spec}D\to L_\Theta$ be a morphism sending the closed point to $Z_0$. If the geometric generic member of the corresponding family is reducible, then the generic point of $Z$ lies in $L_\alpha\cap L_\beta$.
\end{lemma}

\begin{proof}
Let $f_Z$ be a primitive equation over $D$. Its reduction spans $\ker M_\Theta(Z_0)$, so it is a unit times $AB$. After a finite extension of $D$ we can write $f_Z=gh$ with $g,h$ primitive and non-constant, by Gauss's lemma. The reductions $\bar g,\bar h$ are homogeneous of the same degrees as $g,h$. Since $A$ and $B$ are absolutely irreducible and not associated, we get $\{\bar g,\bar h\}=\{A,B\}$ up to units, say $\bar g=A$.

Moreover, Hasse derivatives reduce to Hasse derivatives, so $\mult_zg\le\mult_{\bar z}\bar g$ at every point. Combined with the inequality $\mult_zg+\mult_zh\ge37=\mult_{\bar z}A+\mult_{\bar z}B$, this shows that all these inequalities are equalities. Hence $g$ has class $\alpha$ and $h$ has class $\beta$, which completes the proof.
\end{proof}

\begin{lemma}\label{lem:nodes}
Formally at $n_j$, the family $\cC$ is given by $\xi\eta=\tau_j$ with $\tau_j\in\fm_R$. Moreover, $\tau_j\notin pR$ for every $j$.
\end{lemma}

\begin{proof}
After an unramified extension of $\ZZ_p$, which does not affect the argument, the nodes and their branches are rational. The first assertion is then the versal deformation of a split node.

Suppose now that $\tau_1\in pR$. In $\bar R=R/p\cong\FF_p[[u]]$, the locus $V(\bar\tau_2,\dots,\bar\tau_k)$ has dimension at least $10-(k-1)$. This is bigger than $\dim(L_\alpha\cap L_\beta)_{\FF_p}=10-k$, so there is a curve germ in this locus which is not contained in $L_\alpha\cap L_\beta$. Let $D$ be its normalization. Over $D$ all the nodes persist, so near each $n_j$ the family is $D[[\xi,\eta]]/(\xi\eta)$, and away from the nodes it is smooth over $D$.

Normalizing along the node sections, we obtain a proper smooth family over $D$ with special fibre $\tilde A\sqcup\tilde B$. In such a family the number of geometric connected components is locally constant, since the Stein factorization is finite \'etale. Hence the generic member is reducible, and Lemma~\ref{lem:spec} shows that the germ lies in $L_\alpha\cap L_\beta$, a contradiction.
\end{proof}

\subsection{Passage to characteristic zero}

We choose a lift off the locus preserving both components and off every locus where a node persists.

\begin{proof}[Proof of Theorem~\ref{thm:seed}]
If an unramified extension was used to split the nodes, we make that extension throughout the following argument, replacing $\ZZ_p$ by its ring of integers; we retain the same notation.
Choose a non-zero element $a\in\fa$. Since $R$ is a domain, $F=a\prod_j\tau_j$ is non-zero. A non-zero power series in $\ZZ_p[[u]]$ does not vanish identically on $(p\ZZ_p)^{10}$; this follows from Strassmann's theorem by induction on the number of variables. Hence there is $u^0\in(p\ZZ_p)^{10}$ with $F(u^0)\ne0$, which defines a point $Z\in L_\Theta(\ZZ_p)$ over $Z_0$.

Since $a(Z)\ne0$, Lemma~\ref{lem:spec} shows that $\Gamma_Z$ is geometrically irreducible. By properness, a singular point of $\Gamma_Z$ over $\CC_p$ would specialize to one of the nodes $n_j$. But near $n_j$ the curve $\Gamma_Z$ is given by $\xi\eta=\tau_j(Z)\ne0$, so it is smooth there. Hence $\Gamma_Z$ is smooth. Furthermore, $\ob_\Theta$ has rank $10$ at $Z$ by Lemma~\ref{lem:smooth}.

To pass to $\CC$, consider the finite-type effectivity scheme defined over $\QQ$ by the same interpolation minors.
On the open locus where the kernel is a line, it carries the corresponding universal divisor on the blown-up family.
Smoothness and geometric irreducibility of its fibres, together with the rank-$10$ condition on $\ob_\Theta$, define an open locus after shrinking around the point just constructed; see \cite[9.7.7]{EGA} for geometric irreducibility in families.
Our construction gives a point of this locus over a finite extension of $\QQ_p$ (allowing the unramified extension used to split the nodes).
It is therefore a nonempty finite-type scheme over $\QQ$, and has a point over $\overline{\QQ}$, hence over $\CC$.
The same Jacobian calculation as in Lemma~\ref{lem:smooth} makes this locus smooth of dimension $10$.
This proves the claimed dimension without requiring an embedding of a $p$-adic field into $\CC$.
\end{proof}

\section{Supporting calculations over \texorpdfstring{$\FF_{127}$}{F\_127}}\label{sec:cert}

Let $p=127$ and $\omega=19$, which is a primitive cube root of unity modulo $127$. Let $g=\operatorname{diag}(\omega,\omega^2,1)$. In the affine chart $z=1$ we put
$$
P=(0,0),\quad q_1=(1,1),\quad q_2=(1,68),\quad q_3=(11,61),
$$
and we consider the configuration of ten points
$$
Z_0=\{P\}\cup\langle g\rangle q_1\cup\langle g\rangle q_2\cup\langle g\rangle q_3 .
$$

The coordinates below specify the entire input to the certificate.
The required checks concern dimensions of three kernels, ranks of first-order obstruction maps, and the geometry of two plane factors.
All ranks in this section are computed over $\FF_{127}$.
The search used to find the input is described separately in Appendix~\ref{sec:constraints}; its exhaustiveness is not needed for the proof.

\begin{theorem}\label{thm:cert}
The configuration $Z_0$ satisfies conditions \textup{(S1)--(S3)} of Theorem~\ref{thm:seed} with $k=2$. More precisely, the following hold.
\begin{enumerate}[label=\textup{(\alph*)}]
\item $h^0(\Theta)=1$, $h^1(\Theta)=10$, and $\ob_\Theta$ has rank $10$. The unique member $f_0=AB$ is $g$-invariant.
\item The curve $A$ has degree $12$, multiplicity $5$ at $P$, multiplicity $4$ at the points of $\langle g\rangle q_1\cup\langle g\rangle q_2$, and multiplicity $3$ at the points of $\langle g\rangle q_3$; so $\alpha^2=-4$ and $\tilde A\cong\PP^1$. The curve $B$ has degree $105$ and multiplicities $32$, $33$, $34$ at the same sets of points; so $\beta^2=-1$ and $\tilde B$ has genus $9$.
\item Both $A$ and $B$ are absolutely irreducible. At the ten points they have ordinary multiple points, and they have no other singular points, neither affine nor at infinity.
\item We have $i_P(A,B)=5\cdot32+2$, and $i_z(A,B)=\mult_zA\cdot\mult_zB$ at the remaining nine points; these numbers sum up to $12\cdot105$. The tangent cones of $A$ and $B$ at $P$ have greatest common divisor $xy$. Hence $\tilde A$ and $\tilde B$ meet transversally at two points of $E_P$, and $\alpha\cdot\beta=2$.
\item $h^0(\alpha)=h^0(\beta)=1$, $h^1(\alpha)=3$, $h^1(\beta)=9$, the obstruction maps have ranks $3$ and $9$, and together they have rank $12$.
\end{enumerate}
\end{theorem}

\begin{proof}
All the properties listed above were computed twice, by two independent methods; the results are summarized in Table~\ref{tab:cert}.

In the first method we work with the interpolation matrices modulo $127$. These have size $7030\times7021$ for $\Theta$ and $5679\times5671$ for $\beta$. Their ranks were computed with our own elimination routine and with \textsc{FLINT} \cite{FLINT}; the factorization and the resultants were also computed with \textsc{FLINT}.

The second method does not use interpolation matrices at all, and it was carried out in \textsc{Singular} \cite{Sing}. Since the ideals $\fm_z^{m_z}$ are pairwise comaximal, the fat point ideal $I=\bigcap_z\fm_z^{m_z}$ is equal to the product $\prod_{j\ge1}I(S_j)$, where $S_j=\{z:m_z\ge j\}$ and $I(S_j)$ is the ideal of these points. We computed a standard basis of $I$ and checked that $\dim R/I=7030$. For a degree ordering, the image in $R/I$ of the polynomials of degree at most $d$ is spanned by the standard monomials of degree at most $d$. Therefore $h^0$ and $h^1$ are obtained by counting standard monomials, and the part of degree bigger than $d$ of a normal form realizes the projection $R/I\to\coker M$.

Moreover, since $\partial_xf\in\fm_z^{m_z-1}$, the obstruction of the motion $\partial/\partial x_z$ is represented by the part of degree bigger than $d$ of
$$
\NF\Bigl(\partial_xf\cdot\prod_{w\ne z}L_w^{m_w}\Bigr)\Big/\prod_{w\ne z}L_w(z)^{m_w},
$$
where $L_w$ is a line passing through $w$ and not through $z$. In order to check that $A$ and $B$ have no singular points outside $Z_0$, we compared the length of $R/(A,A_x,A_y)$ with the sum of the local lengths at the ten points. We obtained $81$ for $A$ and $8092$ for $B$ in both cases, and we checked the points at infinity after homogenizing. For the irreducibility check, one first verifies that $A$ and $B$ are irreducible over $\FF_{127}$.
Each also has a smooth $\FF_{127}$-rational point outside $Z_0$.
These two facts together imply absolute irreducibility: Frobenius acts transitively on the geometric components of an irreducible curve, whereas the unique geometric component through a smooth rational point is Frobenius-stable.
A smooth rational point alone would not suffice without the first check.

As a consistency check, the eight infinitesimal projective motions lie in all three kernels. The stacked matrix $[\ob_\Theta;\ob_\alpha;\ob_\beta]$ has rank $12$, which means that $\ker\ob_\alpha\cap\ker\ob_\beta\subset\ker\ob_\Theta$.
\end{proof}

\begin{table}[h]
\centering\small
\begin{tabular}{@{}lccc@{}}
\toprule
 & $\Theta$ & $\alpha$ ($A$, degree $12$) & $\beta$ ($B$, degree $105$)\\
\midrule
length of the fat point scheme & $7030$ & $93$ & $5679$\\
$h^0$, $h^1$ & $1$, $10$ & $1$, $3$ & $1$, $9$\\
rank of the obstruction map & $10$ & $3$ & $9$\\
combined rank of $(\alpha,\beta)$; together with $\Theta$ & \multicolumn{3}{c}{$12$;\quad$12$}\\
Tjurina length: global $=$ sum over $Z_0$ & -- & $81=81$ & $8092=8092$\\
$i(A,B)$: at $P$; in total & \multicolumn{3}{c}{$162=5\cdot32+2$;\quad $1260=12\cdot105$}\\
\bottomrule
\end{tabular}
\caption{The certificate for $Z_0$ over $\FF_{127}$. Both methods give the same values.}\label{tab:cert}
\end{table}

\section{Proof of Theorem A}\label{sec:proofA}

By Theorem~\ref{thm:cert}, the configuration $Z_0$ satisfies all the assumptions of Theorem~\ref{thm:seed}. Hence there exists a complex ten-point blow-up carrying a smooth irreducible curve $\Gamma\in\ls\Theta$ with $\ob\ne0$. By Corollary~\ref{cor:red}, the class $27379H-8658E$ is nef on the blow-up of $\PP^2$ at ten very general points.

Intersecting with a curve $C$ of degree $d$ and multiplicities $m_i$ gives $27379\,d\ge8658\sum_im_i$. This inequality is equivalent to the bound on the Seshadri constant, which completes the proof. \qed

\begin{corollary}\label{cor:counterexample}
A plane curve violating Nagata's inequality at ten very general points has degree at least $8658$ and total multiplicity at least $27379$.
\end{corollary}

\begin{proof}
Put $S=\sum_i m_i$.
Theorem~A and the assumed violation give
$$
\sqrt{10}<\frac{S}{d}\le\frac{27379}{8658}.
$$
Since $4443/1405<\sqrt{10}$ and $1405\cdot27379-4443\cdot8658=1$, we obtain
$$
1\le1405S-4443d\le\frac{d}{8658}.
$$
Thus $d\ge8658$, and $S>\sqrt{10}\,d$ implies $S\ge27379$.
\end{proof}

\section{Scope of the construction}\label{sec:conclusion}

The proof exhibits a way to obtain information about a very general complex surface from a single calculation over a finite field.
The useful output of that calculation is a divisor together with its deformation data.
In our example the divisor is reducible, but its two components impose two more independent conditions than their sum.
This difference makes it possible to leave the reducible locus and smooth both nodes while preserving the total divisor class.
The resulting smooth negative curve then supplies the semistable normal bundle required by the reflection method.

This also explains the role of symmetry.
Cyclic symmetry reduces the search to a tractable finite computation, but the smoothing takes place in the full space of configurations.
At the certified point the locus preserving the two component classes has dimension $8$, accounted for by projective motions, while the total effectivity locus has dimension $10$.
The extra directions are essential to the construction.
The dimension $10$ here is a dimension of ordered configurations; the corresponding local family modulo projective transformations has dimension $2$.

There are two distinct requirements in any further application.
One must choose a negative class whose reflection gives useful numerical information, and one must realize that class by a divisor with sufficiently controlled deformations.
The Pell equation addresses only the first requirement.
The next negative Pell class $2280H-721E$ has expected effectivity codimension $190$, so the present transverse construction in a $20$-dimensional configuration space cannot simply be iterated to that class.
Other classes or different geometric constructions would be needed.
The lifting and smoothing criterion identifies concrete rank and singularity conditions to test in such a search.

\appendix
\section{Symmetry constraints and additional constructions}\label{sec:constraints}

We collect here the constraints and searches which guided the choice of the finite-field example.
They provide context for the construction but are not needed for the implication from the certificate to Theorem~A.

\begin{remark}\label{rem:search}
Let us briefly explain how the configuration $Z_0$ was found. A $\ZZ/3$-invariant curve of class $\Theta$ lies in one of the three eigenspaces of forms of degree $117$ with multiplicity $37$ at the fixed point $P$; for the trivial character this space has dimension $2106$. The three free orbits impose $2109$ conditions, so the $\Theta$-locus is expected to have codimension $4$ in the $4$-parameter family of pairs $(q_2,q_3)$.

We tested all $2\,301\,061$ configurations of this type over $\FF_{127}$, up to a symmetry group of order $108$, which took about $14$ core-hours. We found $3735$ configurations with a unique invariant member, and none of these members is irreducible. In fact, $3713$ of them have a multiple component, $21$ are reduced with three to five components, and only $Z_0$ gives a reduced member with exactly two components.

It is worth noticing that $Z_0$ is a transverse point of the symmetric family, since the corresponding $4\times4$ Jacobian determinant equals $59\ne0$. Moreover, the pair $(A,B)$ is rigid, because the tangent space to $L_\alpha\cap L_\beta$ consists only of the projective motions.
\end{remark}

\subsection{Restrictions on symmetries}

\begin{proposition}\label{prop:aut}
If $\Gamma\in\ls\Theta$ on $X_Z$ is preserved by an automorphism $\sigma$ of $X_Z$, then $\sigma$ is induced by a projective transformation permuting the points of $Z$.
\end{proposition}

\begin{proof}
The pull-back $\sigma^*$ fixes $K$ and $\Theta$, hence it fixes $\Theta+37K=6H$, and so it fixes $H$. Therefore $\sigma$ commutes with the morphism $X_Z\to\PP^2$ given by $\ls H$.
\end{proof}

\begin{proposition}\label{prop:inv}
If $Z$ is invariant under an involution of $\PP^2$, then every member of $\ls\Theta$ is reducible.
\end{proposition}

\begin{proof}
Every involution is conjugate to $\sigma=\operatorname{diag}(-1,1,1)$, with axis $x=0$ and centre $(1:0:0)$. The forms of degree $117$ which are even in $x$ span a space of dimension $3540$. A pair $\{q,\sigma q\}$ imposes at most $703$ conditions on these forms. A point on the axis imposes $361$ conditions, since only even jets matter. The centre imposes $342$ conditions, coming from the jets of odd degree.

If four of the points lie on the axis, then the axis is a component of every member, since $\Theta\cdot(H-E_1-\dots-E_4)<0$. Otherwise, a direct check of the possible orbit types shows that the points impose at most $3534$ conditions. Hence $h^0(\Theta)\ge6$. But a class of negative self-intersection with $h^0\ge2$ has a fixed component, so every member is reducible.
\end{proof}

By Propositions~\ref{prop:aut} and~\ref{prop:inv}, the symmetry group of a pair $(X_Z,\Gamma)$ with $\Gamma$ irreducible is a group of projectivities of odd order.

\begin{remark}[Where to look]
We ran a census of the cyclic types with distinct weights, excluding homologies and types with four collinear points. It leaves two types for which the $\Theta$-locus has non-negative expected dimension and a count of semi-invariants does not already force $h^0(\Theta)\ge2$. These are $\ZZ/5$ with orbits $5+5$, where the expected codimension is $2$ in a $2$-parameter family, and $\ZZ/3$ with orbits $3+3+3+1$ from Section~\ref{sec:cert}.

In the $\ZZ/5$ family we scanned the sixteen primes $p\equiv1\pmod5$ with $131\le p\le491$, and every member of $\ls\Theta$ we found is reducible. The most common splittings have components of degrees $95+11+11$ and $87+30$. The splitting $\theta_2+(57H-18E)$ also occurs, where $\theta_2=60H-19E$ is the class from Remark~\ref{rem:rung4}.
\end{remark}

\begin{remark}[Numerical selection of the reflecting class]
One can ask whether there are other classes which could play the role of $\Theta$. We searched the classes $C=dH-\sum_im_iE_i$ with at most two distinct multiplicities, $d\le200$, $C^2<0$ and expected codimension at most $20$, and we compared the bound given by $s_C(H)$ with Petrakiev's bound $2280/721=3.16227462\ldots$ Only three classes beat it:
\begin{center}
\begin{tabular}{lccc}
\toprule
class & $117H-37E$ & $98H-31E$ & $79H-25E$\\
expected codimension & $10$ & $11$ & $11$\\
bound & $3.16227766$ & $3.16227751$ & $3.16227684$\\
\bottomrule
\end{tabular}
\end{center}
The next Pell class $2280H-721E$ has expected codimension $190$, far above the $12$ moduli of ten points. So going beyond Theorem~A with a single reflection would require a curve much more special than any configuration is expected to carry. For twenty points the analogous search finds no class of expected codimension at most $32$ which beats Petrakiev's bound $720/161$.
\end{remark}

\begin{remark}[Petrakiev's bound from a single elliptic curve]\label{rem:rung4}
Our approach also gives a short proof of Petrakiev's bound. Let $p=131$, let $\zeta$ be a primitive fifth root of unity, and consider
$$
Z=\langle\operatorname{diag}(\zeta,\zeta^{-1},1)\rangle\cdot\{(1:1:1),(49:14:1)\}.
$$
The blow-up $X_Z$ carries a smooth elliptic curve $T\in\ls{60H-19E}$ with $T^2=-10$ and $h^0=1$. This point is Hensel-liftable in the symmetric family, since the Jacobian determinant is $63$, so $T$ lifts to a smooth curve over $\ZZ_{131}$, and hence over $\CC$. Moreover, $h^1(N_T)=10$ and the obstruction map has rank $10$, which persists on the lift.

The destabilizing extensions form a cone of dimension at most $9<10$ in $H^1(N_T)$. So for a general motion the bundle $N_{T/\cX}$ is semistable, and Proposition~\ref{prop:refl} gives the nefness of $s_T(H)=H+12T=721H-228E$. In other words, the bound $d/m\ge2280/721$ follows from the existence of a single curve.
\end{remark}

\section*{AI Declaration}
Symbolic computations revolving around configuration $Z_0$ were designed and optimized with
the use of \texttt{Claude} (Anthropic). All the mathematical results were obtained using human intelligence by both authors.

\section*{Funding}
Piotr Pokora is supported by the National Science Centre (Poland) Sonata Bis Grant  \textbf{2023/\allowbreak 50/\allowbreak E/\allowbreak ST1/\allowbreak 00025}. For the purpose of Open Access, the author has applied a CC-BY public copyright license to any Author Accepted Manuscript (AAM) version arising from this submission.

\section*{Data availability statement}
The scripts and logs used in Section~\ref{sec:cert} are available under \cite{repo}.

\section*{Conflict of interests}
The authors have no conflicts of interest to disclose.


\begin{thebibliography}{99}
\bibitem{Primer} Th.~Bauer, S.~Di Rocco, B.~Harbourne, M.~Kapustka, A.~L.~Knutsen, W.~Syzdek, and T.~Szemberg, A primer on Seshadri constants. In: \emph{Interactions of classical and numerical algebraic geometry}, Contemp. Math. \textbf{496}: 33--70 (2009).
\bibitem{Bir} P.~Biran, Constructing new ample divisors out of old ones. \emph{Duke Math. J.} \textbf{98}: 113--135 (1999).
\bibitem{CDMR} C.~Ciliberto, O.~Dumitrescu, R.~Miranda, and J.~Ro\'e, Emptiness of homogeneous linear systems with ten general base points. In: \emph{Classification of algebraic varieties}, EMS Ser. Congr. Rep., European Mathematical Society, Z\"urich, 189--195 (2011).
\bibitem{CM} C.~Ciliberto and R.~Miranda, Homogeneous interpolation on ten points. \emph{J. Algebraic Geom.} \textbf{20}: 685--726 (2011).
\bibitem{Sing} W.~Decker, G.-M.~Greuel, G.~Pfister, and H.~Sch\"onemann, \textsc{Singular} 4-3-2 --- A computer algebra system for polynomial computations. \url{https://www.singular.uni-kl.de} (2023).
\bibitem{dFneg} T.~de Fernex, Negative curves on very general blow-ups of $\PP^2$. In: \emph{Projective varieties with unexpected properties}, de Gruyter, Berlin, 199--207 (2005).
\bibitem{dFcone} T.~de Fernex, On the Mori cone of blow-ups of the plane. Preprint, \href{https://arxiv.org/abs/1001.5243}{arXiv:1001.5243} (2010).
\bibitem{DR} C.~Dionne and M.~Roth, Seshadri constants on $\PP^1\times\PP^1$ and applications to the symplectic packing problem. \emph{Forum Math. Sigma} \textbf{13}: e191, 42 pp. (2025).
\bibitem{Dum} M.~Dumnicki, Regularity and non-emptyness of linear systems in $\PP^n$. Preprint, \href{https://arxiv.org/abs/0802.0925}{arXiv:0802.0925} (2008).
\bibitem{Eck} T.~Eckl, Ciliberto--Miranda degenerations of $\CC\PP^2$ blown up in 10 points. \emph{J. Pure Appl. Algebra} \textbf{215}: 672--696 (2011).
\bibitem{FLINT} The FLINT team, \textsc{FLINT}: Fast Library for Number Theory, version 3. \url{https://flintlib.org}.
\bibitem{EGA} A.~Grothendieck, \'El\'ements de g\'eom\'etrie alg\'ebrique IV. \'Etude locale des sch\'emas et des morphismes de sch\'emas III. \emph{Publ. Math. IH\'ES} \textbf{28}: 5--255 (1966).
\bibitem{HS} T.~H.~H\`a and A.~Sivakumar, Polynomial interpolation and the Waldschmidt constant of points in projective space. Preprint, \href{https://arxiv.org/abs/2608.16040}{arXiv:2608.16040} (2026).
\bibitem{HR} B.~Harbourne and J.~Ro\'e, Computing multi-point Seshadri constants on $\PP^2$. \emph{Bull. Belg. Math. Soc. Simon Stevin} \textbf{16}: 887--906 (2009).
\bibitem{LU} A.~Laface and L.~Ugaglia, Irrational Seshadri constants from dihedral orbits. Preprint, \href{https://arxiv.org/abs/2609.26521v2}{arXiv:2609.26521v2} (2026).
\bibitem{Nag} M.~Nagata, On the 14-th problem of Hilbert. \emph{Amer. J. Math.} \textbf{81}: 766--772 (1959).
\bibitem{Pet} I.~Petrakiev, Homogeneous interpolation and some continued fractions. \emph{Trans. Amer. Math. Soc. Ser. B} \textbf{1}: 23--44 (2014).
\bibitem{repo}
P.~Pokora and T.~Szemberg, Lifting negative curves from finite fields and Seshadri constants at ten points: scripts and logs. GitHub repository, \url{https://github.com/piotrpokora/lifting-negative-curves-seshadri10} (2026).
\bibitem{Ser} E.~Sernesi, \emph{Deformations of algebraic schemes}. Grundlehren der Mathematischen Wissenschaften \textbf{334}, Springer-Verlag, Berlin, 2006.
\end{thebibliography}
\end{document}